\documentclass[12pt,leqno]{article}
\usepackage{amsfonts}
\usepackage{amsmath, amsthm, amsfonts, amssymb, color}
\usepackage{mathrsfs}
\usepackage{color}
\usepackage{stmaryrd}
\makeatletter
\newcommand{\Spvek}[2][r]{%
  \gdef\@VORNE{1}
  \left(\hskip-\arraycolsep%
    \begin{array}{#1}\vekSp@lten{#2}\end{array}%
  \hskip-\arraycolsep\right)}

\def\vekSp@lten#1{\xvekSp@lten#1;vekL@stLine;}
\def\vekL@stLine{vekL@stLine}
\def\xvekSp@lten#1;{\def\temp{#1}%
  \ifx\temp\vekL@stLine
  \else
    \ifnum\@VORNE=1\gdef\@VORNE{0}
    \else\@arraycr\fi%
    #1%
    \expandafter\xvekSp@lten
  \fi}
\makeatother

\newtheorem{thm}{Theorem}[section]
\newtheorem{cor}[thm]{Corollary}

\newtheorem{rem}[thm]{Remark}
\theoremstyle{definition}

\newcommand{\scr}[1]{\mathscr #1}
\definecolor{wco}{rgb}{0.5,0.2,0.3}

\numberwithin{equation}{section} \theoremstyle{remark}

\newcommand{\ua}{\uparrow}

\title{{\bf   Harnack and Shift Harnack Inequalities for Degenerate (Functional) SPDEs with Singular Drifts}
}
\author{
{\bf     Xing Huang$^{a)}$, Wujun Lyu $^{b)}$   }\\
\footnotesize{  a)Center of Applied Mathematics, Tianjin University, Tianjin 300072, China}\\
\footnotesize{  xinghuang@tju.edu.cn}\\
\footnotesize{$^{b)}$         School of Mathematics, Shanghai University of Finance and Economics, Shanghai 200433, China}\\
\footnotesize{  lvwujunjaier@gmail.com}}
\begin{document}
\allowdisplaybreaks
\def\R{\mathbb R}  \def\ff{\frac} \def\ss{\sqrt} \def\B{\mathbf
B}
\def\N{\mathbb N} \def\kk{\kappa} \def\m{{\bf m}}
\def\ee{\varepsilon}\def\ddd{D^*}
\def\dd{\delta} \def\DD{\Delta} \def\vv{\varepsilon} \def\rr{\rho}
\def\<{\langle} \def\>{\rangle} \def\GG{\Gamma} \def\gg{\gamma}
  \def\nn{\nabla} \def\pp{\partial} \def\E{\mathbb E}
\def\d{\text{\rm{d}}} \def\bb{\beta} \def\aa{\alpha} \def\D{\scr D}
  \def\si{\sigma} \def\ess{\text{\rm{ess}}}
\def\beg{\begin} \def\beq{\begin{equation}}  \def\F{\scr F}
\def\Ric{\text{\rm{Ric}}} \def\Hess{\text{\rm{Hess}}}
\def\e{\text{\rm{e}}} \def\ua{\underline a} \def\OO{\Omega}  \def\oo{\omega}
 \def\tt{\tilde} \def\Ric{\text{\rm{Ric}}}
\def\cut{\text{\rm{cut}}} \def\P{\mathbb P} \def\ifn{I_n(f^{\bigotimes n})}
\def\C{\scr C}   \def\G{\scr G}   \def\aaa{\mathbf{r}}     \def\r{r}
\def\gap{\text{\rm{gap}}} \def\prr{\pi_{{\bf m},\varrho}}  \def\r{\mathbf r}
\def\Z{\mathbb Z} \def\vrr{\varrho} \def\ll{\lambda}
\def\L{\scr L}\def\Tt{\tt} \def\TT{\tt}\def\II{\mathbb I}
\def\i{{\rm in}}\def\Sect{{\rm Sect}}  \def\H{\mathbb H}
\def\M{\scr M}\def\Q{\mathbb Q} \def\texto{\text{o}} \def\LL{\Lambda}
\def\Rank{{\rm Rank}} \def\B{\scr B} \def\i{{\rm i}} \def\HR{\hat{\R}^d}
\def\to{\rightarrow}\def\l{\ell}\def\iint{\int}
\def\EE{\scr E}\def\no{\nonumber}
\def\A{\scr A}\def\V{\mathbb V}\def\osc{{\rm osc}}
\def\BB{\scr B}\def\Ent{{\rm Ent}}\def\3{\triangle}\def\H{\scr H}
\def\U{\scr U}\def\8{\infty}\def\1{\lesssim}\def\HH{\mathrm{H}}
 \def\T{\scr T}
\maketitle

\begin{abstract} The existence and uniqueness of the mild solutions for a class of degenerate functional SPDEs are obtained, where the drift is assumed to be H\"{o}lder-Dini continuous. Moreover, the non-explosion of the solution is proved under some reasonable conditions. In addition, the Harnack is derived by the coupling by change of measure. Finally, the shift Harnack inequality is obtained for the equations without delay, which is new even in the non-degenerate case.
\end{abstract} \noindent
 \vskip 2cm

\section{Introduction}
The stochastic Hamiltonian system is an important model of degenerate diffusion system, which has been investigated in \cite{GW,W2,WZ1,WZ,Z1}. In this paper we aim to study functional version of this model (see \cite{BWY}) in infinite dimension. Recently, Zvonkin type transforms have been used to prove existence and uniqueness of SDEs and SPDEs with singular drift, see e.g. \cite{B,GM,P,FYW,ZV,Z}. In \cite{HW}, the author has investigated the non-degenerate functional SPDEs with Dini continuous drift, see also \cite{H} for the finite dimensional non-degenerate functional SDEs with integrable drifts. The purpose of this paper is to investigate the degenerate functional SPDEs with singular coefficients. We adopt the Zvonkin type transforms considered in \cite{WZ15} for SPDEs which enable us to regularize a singular drift without time delay. Therefore, in the functional SPDEs, we only allow the drift to be H\"{o}lder-Dini continuous in the present state, drift with delay being Lipschitz continuous, see {\bf(a4)} below. On the other hand, Harnack and shift Harnack inequalities have many applications, see \cite[Chapter 1]{Wbook}. \cite{FYW} prove the Log-Harnack inequalities for SPDEs with Dini drift by gradient estimate. Following this, using coupling by change of measure, \cite{HW} obtain the Log-Harnack inequalities for functional SDEs with Dini continuous drift. However, the noises are non-degenerate in both of these two cases. So far, there is no result on the Harnack or shift Harnack inequalities for degenerate SPDEs. In this paper, we construct coupling by change of measure to derive the Harnack and shift Harnack inequalities for a class of functional SPDEs with singular drift.

Let $\mathbb{H}_i$ (i=1,2,3) be three separable Hilbert spaces.
Let $\L\left(\mathbb{H}_i; \mathbb{H}_j\right)$ be the space of bounded linear operators from $\mathbb{H}_i$ to $\mathbb{H}_j$ ($1\leq i,j\leq 3$). For simplicity, we denote the norm and inner product by $|\cdot|$ and $\langle\cdot,\cdot\rangle$ for Hilbert spaces, and let $\|\cdot\|$ stand for the operator norm. Let $\mathbb{H}=\mathbb{H}_1\times\mathbb{H}_2$. For $\mathbb{H}\ni z=(x,y)\in\mathbb{H}_1\times\mathbb{H}_2$, $|z|^2:=|x|^2+|y|^2$.
To describe the time delay, let $\nu$ be a non-trivial  measure on $(-\infty,0)$ such that
\beq\label{nu} \nu \ \text{is\ locally\ finite\ and}\   \nu(\cdot-t)\le \kk(t)\nu(\cdot), \ \  t> 0\end{equation}   for some  increasing  function $\kk: (0,\infty)\to (0,\infty)$.
This condition is crucial to prove the pathwise (see the proof of Proposition  \ref{P-uni} below), and to determine  the state space of the segment solutions.  Obviously, \eqref{nu} holds for $\nu(\d\theta):=1_{(-\infty,0)}(\theta)\rr(\theta)\d\theta$ with density $\rr\ge 0$ satisfying
$\rr(\theta-t)\le\kk(t)\rr(\theta), t>0$ for $\theta<0,$ which is the case if, for instance, $\rr(\theta)= \e^{\ll\theta} 1_{[-r_0,0)}(\theta)$ for some constants
$\ll\in\R$ and $r_0\in (0,\infty).$
Then the  state space of the segment process under study  is given by
$$\C_\nu:= \bigg\{\xi: (-\infty,0]\to \mathbb{H} \ \text{is\ measurable\ with}\ \nu(|\xi|^2)<\infty\bigg\},$$
 where $\nu(f):=\int_{-\infty}^0 f(\theta)\nu(\d\theta)$ for $f\in L^1(\nu).$
 Let
\begin{align}\label{Cnu}\|\xi\|_{\C_\nu}:=\ss{ \nu (|\xi|^2) + |\xi(0)|^2},\ \ \xi\in \C_\nu.
\end{align}
Throughout the paper, we   identify  $\xi$ and $\eta$ in $\C_\nu$ if   $\xi=\eta\ \nu$-a.e. and $\xi(0)=\eta(0)$, so that
$\C_\nu$ is a separable Hilbert space with inner product
$$\<\xi,\eta\>_{\C_\nu}:= \nu(\<\xi,\eta\>) +\xi(0)\eta(0),\ \ \xi, \eta\in \C_\nu.$$
For a map $X: \R\to \mathbb{H}$ and $t\ge 0$, let $X_t: (-\infty, 0]\to \mathbb{H}$ be defined by
$$X_t(\theta)= X(t+\theta),\ \ \theta\in (-\infty,0].$$  We call $X_t$ the segment of $X$ at time $t$.
Similarly, for $i=1,2$, we denote $\C^i_\nu$ for $\mathbb{H}_i$ instead of $\mathbb{H}$. Clearly, for $\xi_i\in\C^i_\nu,i=1,2$, let $\xi(s)=(\xi_1(s),\xi_2(s)), s\in(-\infty,0]$, then $\xi\in\C_\nu$.

Different from the finite dimension case, the main difficulty in the proof of the pathwise uniqueness is that the new equation after Zvonkin transform contains an unbounded drift $-A_2u$, see \eqref{regular} below. To treat this term, we use the Fubini theorem and use $\|\cdot\|_{\C_\nu}$ for the delay part instead of the usual uniform norm or weighted uniform norm.

Let $W=(W(t))_{t\geq 0}$ be a cylindrical Brownian motion on $\mathbb{H}_3$ with respect to a complete filtration probability space $(\OO, \F, \{\F_{t}\}_{t\ge 0}, \P)$. More precisely, $W(\cdot)=\sum_{n=1}^{\infty}{\bar{W}^{n}(\cdot)h_{n}}$ for a sequence of independent one dimensional standard Brownian motions $\left\{\bar{W}^{n}(\cdot)\right\}_{n\geq 1}$ with respect to $(\OO, \F,
\{\F_{t}\}_{t\ge 0}, \P)$, where $\{h_{n}\}_{n\geq 1}$ is an orthonormal basis on $\mathbb{H}_3$.

Consider the following functional SPDE on $\mathbb{H}$:
\beq\label{IE1}
\begin{cases}
\d X(t)= \{A_1 X(t)+B Y(t)\}\d t, \\
\d Y(t)=\{A_2 Y(t)+b(t,X(t),Y(t))+F(t,X_t,Y_t)\}\d t+Q(t)\d W(t),
\end{cases}
\end{equation}
where $B\in\L\left(\mathbb{H}_2; \mathbb{H}_1\right)$, for any $i=1,2$, $(A_i,\D(A_i))$ is a bounded above linear operator generating a strongly continuous semigroup $\e^{tA_i}$ on $\mathbb{H}_i$, $F: [0,\infty)\times \C_\nu\to \mathbb{H}_2$, $b: [0,\infty)\times \mathbb{H}\to \mathbb{H}_2$ and $Q: [0,\infty)\to \L\left(\mathbb{H}_3; \mathbb{H}_2\right)$ are measurable and locally bounded (i.e. bounded on bounded sets). We will still use the same notations as in the finite dimension case, i.e.  $\nabla,\nabla^{(1)}$ and $\nabla^{(2)}$ denote
the gradient operators on $\mathbb{H}$, $\mathbb{H}_1$ and $\mathbb{H}_2$ respectively.

In general, the mild solution (if exists) to \eqref{IE1} can be explosive, so we consider mild solutions with life time.

\beg{defn}\label{D-solution'} A continuous $\mathbb{H}$-valued process $(X(t),Y(t))_{t\in(-\infty,\zeta)}$ is called a mild solution to \eqref{IE1} with life time $\zeta$, if the segment process $(X_t,Y_t)$ is $\F_t$-measurable, and $\zeta>0$ is a stopping time such that $\mathbb{P}$-a.s $\limsup_{t\uparrow\zeta}(|X(t)|+|Y(t)|)=\infty$ holds on $\{\zeta<\infty\}$, and $\mathbb{P}$-a.s
\begin{equation*}
\begin{cases}
X(t)= \e^{A_1(t\vee0)} X(t\wedge0)+\int_{0}^{t\vee0}\e^{A_1(t-s)}B Y(s)\d s, \\
Y(t)=\e^{A_2(t\vee0)} Y(t\wedge0)+\int_{0}^{t\vee0}\e^{A_2(t-s)}\{b(s,X(s),Y(s))+F(s,X_s,Y_s)\}\d s\\
\qquad\qquad+\int_{0}^{t\vee0}\e^{A_2(t-s)}Q(s)\d W(s),\ \  t\in[-r,\zeta).
\end{cases}
\end{equation*}
\end{defn}
Throughout this paper, let $\{e_{n}\}_{n\geq 1}$ be an orthonormal basis on $\mathbb{H}_2$. For any $n \geq 1$, let $\mathbb{H}^{(n)}_2:=\mathrm{span}\{e_1,\cdots, e_n\}$, $\pi^{(n)}_2$ be the orthogonal projection map from $\mathbb{H}_2$ to $\mathbb{H}^{(n)}_2$. Moreover, let $\mathbb{H}^{(n)}_1:=B\mathbb{H}^{(n)}_2$ and $\pi^{(n)}_1$ be the orthogonal projection map from $\mathbb{H}_1$ to $\mathbb{H}^{(n)}_1$. If $BB^\ast$ is invertible, we have $\lim_{n\to\infty}\pi^{(n)}_1x = x$ for $x\in\mathbb{H}_1$. Let
$$ \pi^{(n)}= (\pi^{(n)}_1, \pi^{(n)}_2) :\mathbb{H}^{(n)}:= \mathbb{H}^{(n)}_1\times\mathbb{H}^{(n)}_2.$$
In this section, we investigate the existence and uniqueness of \eqref{IE1} when $b$ is singular. To this end, we need the following assumptions, see \cite{WZ15} for details.
\beg{enumerate}
\item[{\bf (a1)}] $(-A_2)^{\varepsilon-1}$ is of trace class for some $\varepsilon \in(0,1)$; i.e. $\sum_{n=1}^{\infty}{\lambda_{n}^{\varepsilon-1}}<\infty$ for $0< \lambda_{1}\leq \lambda_{2}\leq\cdots$ being all eigenvalues of $-A_2$ counting multiplicities. The eigenbasis of $-A_2$ on $\mathbb{H}_2$ corresponding to the eigenvalues $\{\lambda_i\}_{i=1}^\infty$ is $\{e_i\}_{i=1}^\infty$.

\item[{\bf (a2)}]
\beg{enumerate}
\item[(i)] $Q\in C([0,\infty);\L(\mathbb{H}_3; \mathbb{H}_2))$ such that for every $t\geq 0$, $(Q Q^{\ast})(t)$ is invertible and $\|(Q Q^{\ast})^{-1}(t)\|$ is locally bounded in $t\geq 0$.

\item[(ii)] $BB^\ast$ is invertible in $\mathbb{H}_1$, and $B\e^{A_2 t} = \e^{A_1 t} \e^{A_0 t} B$ for some $A_0 \in \L(\mathbb{H}_1,\mathbb{H}_1)$ and all $t\geq 0$.

\item[(iii)]There exists $n_0 \geq1$ such that for any $n \geq n_0$, $\pi^{(n)}_1 B = B\pi^{(n)}_2$ on $\mathbb{H}_2$, and $\pi^{(n)}_1 A_1 =
A_1\pi^{(n)}_1$ on $\D(A_1)$.
\end{enumerate}
\end{enumerate}

To describe the singularity of $b$, we introduce
\beg{equation*}\beg{split}
\D= \Big\{\phi: [0,\infty)\to [0,\infty) \text{ is increasing}, \phi^{2} \text{ is concave}, \int_0^1{\frac{\phi(s)}{s}\d s}<\infty\Big\}
\end{split}\end{equation*}
\begin{enumerate}
\item[{\bf (a3)}] For any $n\geq1$, there exists $\phi_{n}\in\D$ and constants  $K_n>0$, $\alpha_n\in(\frac{2}{3},1)$ such that
\beq\label{ab}
\sup_{t\in[0,n], (x,y)\in \mathbb{H}, |(x,y)|\leq n}|b(t,x,y)|<\infty,
\end{equation}
and for any $ t\in[0,n], |(x,y)|\vee |(x',y')|\leq n$, it holds that
\beq\label{b-0}
|b(t,x,y)-b(t,x',y')|\leq \phi_{n}(|y-y'|)+K_n|x-x'|^{\alpha_n}.
\end{equation}
\end{enumerate}
\beg{rem}\label{1.1}
Obviously, the class $\D$ contains $\phi(s):=\frac{K}{\log^{1+\delta}(c+s^{-1})}$ for constants $K, \delta >0$ and large enough $c\geq \e$ such that $\phi^{2}$ is concave.
\end{rem}
Moreover, we need the following assumption on the delay part.
\beg{enumerate}
\item[{\bf (a4)}] $F\in C([0,\infty)\times \C_\nu; \mathbb{H}_2)$, and there exists an increasing function $C_F:[0,\infty)\to[0,\infty)$ such that for any $n\geq 1$,
\beg{equation*}\beg{split}
|F(t,\xi)-F(t,\eta)|\leq C_F(n)\|\xi-\eta\|_{\C_\nu},\ \ t\in[0,n], \xi,\eta\in\C_\nu, \|\xi\|_{\C_\nu}\vee\|\eta\|_{\C_\nu}\leq n.
\end{split}\end{equation*}
\end{enumerate}
When supp$\,\nu$ is $\{0\}$, i.e. the case without delay, {\bf (a1)}-{\bf (a3)} imply the existence and uniqueness of the mild solution to \eqref{IE1} by \cite[Theorem 1.1 (1)]{WZ15}.

Throughout the paper, the letter $C$ or $c$ will denote a positive constant, and $C(\theta)$ or $c(\theta)$ stands for a constant depending on $\theta$. The value of the constants may change from one appearance to another.

The paper is organized as follows: In Section 2, we prove the existence, uniqueness and non-explosion of the solution for degenerate functional SPDEs with singular drift; In Section 3, we investigate Harnack inequalities; In Section 4, we investigate the shift Harnack inequalities.
\section{Existence and Uniqueness}
The following theorem gives results on existence, uniqueness and non-explosion of the mild solution.
\beg{thm}\label{T-EU-NE} Assume {\bf (a1)}-{\bf (a4)}.
\beg{enumerate}
\item[$(1)$] For any $\F_0$-measurable initial value $(X_0,Y_0)$, the equation \eqref{IE1} has a unique mild solution $(X(t),Y(t))_{t\in(-\infty,\zeta)}$ with life time $\zeta$.
\item[$(2)$] If there exist two positive functions $\Phi, h:[0,\infty)\times[0,\infty)\to(0,\infty)$ increasing in each variable such that $\int_{1}^{\infty}\frac{\d s}{\Phi_{t}(s)}=\infty$ for any $t\geq0$ and
\beq\label{Bb-phi}\begin{split}
&\langle F(t,\xi,\eta+\eta')+b(t,\xi(0),(\eta+\eta')(0)),\eta(0)\rangle \\
&\leq\Phi_{t}\left(\|\xi\|_{\C^1_\nu}^{2}+\|\eta\|_{\C^2_\nu}^{2}\right)+h_{t}(\|\eta'\|_{\C^2_\nu}), \ \ \xi\in\C^1_\nu, \eta,\eta'\in\C^2_\nu, t\geq 0,
\end{split}\end{equation}
then the mild solution is non-explosive.
\end{enumerate}
\end{thm}
To apply Zvonkin type transform, we in fact need some global version of {\bf (a3)}-{\bf (a4)}, and Theorem \ref{T-EU-NE} can be proved by localization method.

For any $T>0$, let $\|\cdot\|_{T,\infty}$ denote the uniform norm on $[0,T]\times\mathbb{H}$ or $[0,T]\times\C_\nu$.
\beg{enumerate}
\item[$\bf{(a3^{'})}$] For any $T>0$, there exists $\phi\in\D$ and constants  $K>0$, $\alpha\in(\frac{2}{3},1)$ such that
\beq\label{b-T-a}
\| b\|_{T,\infty}<\infty,
\end{equation}
and
\beq\label{b-phi0}
|b(t,x,y)-b(t,x',y')|\leq K|x-x'|^{\alpha}+\phi(|y-y'|),\quad t\in[0,T], x,x'\in\mathbb{H}_1, y,y'\in\mathbb{H}_2.
\end{equation}
\item[$\bf{(a4^{'})}$] For any $t\geq 0$, $\|F\|_{t,\infty}<\infty$. $F$ satisfies {\bf (a4)}, and there exists an increasing function $C_F^{'}:[0,\infty)\to[0,\infty)$ such that for any $T\geq 0$,
\beg{equation*}\beg{split}
|F(t,\xi)-F(t,\eta)|\leq C_F^{'}(T)\|\xi-\eta\|_{\C_\nu},\ \ t\in[0,T], \xi,\eta\in\C_\nu.
\end{split}\end{equation*}
\end{enumerate}
\subsection{Regularization transform}
In this subsection, we transform \eqref{IE1} to a regular equation, the pathwise uniqueness of which is equivilant to that of \eqref{IE1}. To this end, the regularity of the solution to the equation \eqref{u0} (Lemma \ref{L-PDE}) is crucial, which has been proved in  \cite{WZ15}.

Consider the following SPDEs:
\beq\label{E-A-Q}
\begin{cases}
\d X^0_{s,t}(x,y)= \{A_1 X^0_{s,t}(x,y)+B Y^0_{s,t}(x,y)\}\d t, \\
\d Y^0_{s,t}(x,y)=A_2 Y^0_{s,t}(x,y)\d t+Q(t)\d W(t),\ \ (X^0_{s,s}, Y^0_{s,s})(x,y)=(x,y), t\geq s\geq 0.
\end{cases}
\end{equation}

Under {\bf (a1)}, ${\bf (a2)}$, \eqref{E-A-Q} has a unique mild solution $\{(X^0_{s,t}, Y^0_{s,t})(x,y)\}_{t\geq s}$. Let $P_{s,t}^{0}$ be the associated Markov semigroup, i,e.
\begin{equation*}
P_{s,t}^{0}f(x,y)=\mathbb{E}f(X^0_{s,t}(x,y), Y^0_{s,t}(x,y)), \ \ f\in\B_b(\mathbb{H}), t\geq s\geq 0.
\end{equation*}
As in \cite{WZ15}, to transform \eqref{IE1} to a regular equation, we need to study the regularity of the following equation:
\beq\label{u0}
u(s,\cdot)=\int_{s}^{T} \e^{-\lambda(t-s)}P_{s,t}^{0}(\nabla^{(2)}_{b(t,\cdot)}u(t,\cdot)+b(t,\cdot))\d t, \ \ s\in[0,T].
\end{equation}
The following Lemma is from \cite[Proposition 3.1]{WZ15}.
\beg{lem}\label{L-PDE} Assume {\bf (a1)}, ${\bf (a2)}$,  $\bf{(a3^{'})}$. Let $T>0$ be fixed. Then there exists a constant $\lambda_0>0$ such that for any $\lambda\geq\lambda_0$, the equation \eqref{u0} has a unique solution $u^\lambda\in C([0,T]; C_{b}^{1}(\mathbb{H}; \mathbb{H}_{2}))$ and
\beq\label{g1}
\lim_{\lambda\to \infty}\| u^\lambda\|_{T,\infty}+\|\nabla  u^\lambda\|_{T,\infty}+\left\|\nabla\nabla^{(2)} u^\lambda\right\|_{T,\infty}=0.
\end{equation}
\end{lem}
The next lemma gives a regular representation of \eqref{IE1}.
\beg{lem}\label{L-regular} Assume {\bf (a1)}, ${\bf (a2)}$, $b\in\B_b([0,T]; C_b(\mathbb{H},\mathbb{H}_2))$ and $F\in\B_b([0,T]; C_b(\C_\nu,\mathbb{H}_2))$ for some $T\geq 0$ . If $\{Z(t)\}_{t\in(-\infty,T\wedge\tau]}=\{X(t),Y(t)\}_{t\in(-\infty,T\wedge\tau]}$ solves \eqref{IE1} for some stopping time $\tau$, i.e. $\mathbb{P}$-a.s.
\begin{equation*}
\begin{cases}
X(t)= \e^{A_1(t\vee0)} X(t\wedge0)+\int_{0}^{t\vee0}\e^{A_1(t-s)}B Y(s)\d s, \\
Y(t)=\e^{A_2(t\vee0)} Y(t\wedge0)+\int_{0}^{t\vee0}\e^{A_2(t-s)}\{b(s,X(s),Y(s))+F(s,X_s,Y_s)\}\d s\\
\qquad\qquad+\int_{0}^{t\vee0}\e^{A_2(t-s)}Q(s)\d W(s),\ \  t\in(-\infty,T\wedge\tau],
\end{cases}
\end{equation*}
then for any $\lambda\geq\lambda_0$, there holds $\mathbb{P}$-a.s.
\beq\label{regular} \beg{split}
Y(t)&= \e^{A_2t} [Y(0)+u^\lambda(0,Z(0))]-u^\lambda(t,Z(t))\\
&+\int_{0}^{t}(\lambda-A_2)\e^{A_2(t-s)}u^\lambda(s,Z(s))\d s\\
&+\int_{0}^{t}\e^{A_2(t-s)}[\mathrm{I}_{\mathbb{H}_2}+\nabla^{(2)} u^\lambda(s,Z(s))]F(s,Z_{s})\d s\\
&+\int_{0}^{t} \e^{A_2(t-s)}[\mathrm{I}_{\mathbb{H}_2}+\nabla^{(2)} u^\lambda(s,Z(s))]Q(s)\d W(s), \quad t\in[0,\tau\wedge T],
\end{split}\end{equation}
where $\mathrm{I}_{\mathbb{H}_2}$ stands for the identical operator on $\mathbb{H}_2$, $u^\lambda$ solves \eqref{u0}, and $\nabla^{(2)} u(s,z)v:=[\nabla^{(2)}_{v} u(s,\cdot)](z)$ for $v\in\mathbb{H}_2$, $z\in\mathbb{H}$.
\end{lem}
\begin{proof}
Since $F\in\B_b([0,T]; C_b(\C_\nu,\mathbb{H}_2))$, simulating the proof of \cite[Theorem 4.1]{WZ15}, it is easy to obtain the result as desired. To save space, we omit the detail here.
\end{proof}
Now, we present a complete proof of the pathwise uniqueness to \eqref{IE1}.
\beg{prp}\label{P-uni} Assume {\bf (a1)}, {\bf (a2)}, ${\bf (a3^{'})}$-${\bf (a4^{'})}$. Let $\{Z_t\}_{t\geq 0}:=\{(X_t,Y_t)\}_{t\geq 0},\{\tilde{Z}_t\}_{t\geq 0}:=\{(\tilde{X}_t,\tilde{Y}_t)\}_{t\geq 0}$ be two adapted continuous $\C_\nu$-valued processes with $Z_{0}=\tilde{Z}_{0}=\xi\in\C_\nu$. For any $m\geq 1$, let
\beg{equation*}
\tau_{m}^{Z}=m\wedge \inf\{t\geq 0:|Z(t)|\geq m\}, \ \ \tau_{m}^{\tilde{Z}}=m\wedge \inf\{t\geq 0:|\tilde{Z}(t)|\geq m\}.
\end{equation*}
If $Z(t)$ and $\tilde{Z}(t)$ are mild solutions to \eqref{IE1} for $t\in(-\infty,\tau_{m}^{Z}\wedge\tau_{m}^{\tilde{Z}}]$,
then $\mathbb{P}$-a.s. $Z(t)=\tilde{Z}(t)$, for all $t\in(-\infty,\tau_{m}^{Z}\wedge\tau_{m}^{\tilde{Z}}]$. In particular, $\mathbb{P}$-a.s. $\tau_{m}^{Z}=\tau_{m}^{\tilde{Z}}$ for $m\geq 1$.
\end{prp}
\beg{proof}[Proof] For any $m\geq 1$, let $\tau_{m}=\tau_{m}^{Z}\wedge\tau_{m}^{\tilde{Z}}$. It suffices to prove that for any $T>0$ and $m\geq1$,
\beq\label{3.5}
\int_{0}^{T} \mathbb{E}\big\{1_{\{s<\tau_{m}\}}|Z(s)-\tilde{Z}(s)|^{2}\big\}\d s=0.
\end{equation}
Let $\lambda\geq\lambda_0$ be such that assertions in Lemma \ref{L-PDE} and Lemma \ref{L-regular} hold.
By \eqref{regular} for $\tau=\tau_{m}$, we have $\mathbb{P}$-a.s.
\beq\label{X1}Y(t)-\tilde{Y}(t)= \LL(t)+\Xi(t),\ \ t\in [0,\tau_m\land T],\end{equation}
where
\beg{equation*} \beg{split}\LL(t)&:= \int_{0}^{t}\e^{A_2(t-s)}\big\{[\mathrm{I}_{\mathbb{H}_2}+\nabla^{(2)} u^\lambda(s,Z(s))]F(s,Z_{s})\\
&\qquad\qquad\qquad\qquad-[\mathrm{I}_{\mathbb{H}_2}+\nabla^{(2)} u^\lambda(s,\tilde{Z}(s))]F(s,\tilde{Z}_{s})\big\}\d s,\\
\Xi(t)&:= u(t,Z(t))-u(t,\tilde{Z}(t))\\
+&\int_{0}^{t}(\lambda-A_2)\e^{A_2(t-s)}[u^\lambda(s,Z(s))-u^\lambda(s,\tilde{Z}(s))]\d s\\
+&\int_{0}^{t} \e^{A_2(t-s)}[\nabla^{(2)} u^\lambda(s,Z(s))-\nabla^{(2)} u^\lambda(s,\tilde{Z}(s))]Q(s)\d W(s),\ \ t\in[0,\tau_{m}\wedge T].\end{split}\end{equation*}
According to the proof of \cite[Corollary 4.2]{WZ15}, when $\ll\ge\ll_0$ is large enough there exists a constant $C_0\in (0,\infty)$ such that
\beq\label{X2}  \int_0^r \e^{-2\ll t} \E\big[1_{\{t<\tau_m\}} |\Xi(t)|^2\big]\,\d t\le \ff 3 4 \GG(r) + C_0 \int_0^r \GG(t)\d t,\ \ r\in [0,T]\end{equation} holds for
\beq\label{X3} \GG(t):= \int_0^t \e^{-2\ll s}  \E\big[1_{\{s<\tau_m\}} |Z(s)-\tilde{Z}(s)|^2\big]\,\d s,\ \ t\in [0,T].\end{equation}
So, to prove \eqref{3.5}, it remains to estimate the corresponding term for $\LL(t)$ in place of $\Xi(t).$
Noting that $Z_0=\tilde{Z}_0$ in $\C_\nu$ implies $Z=\tilde{Z}\ \nu$-a.e. on $(-\infty,0)$, by \eqref{nu} we have
$$\int_{-\infty}^{-s} |Z(s+q)-\tilde{Z}(s+q)|^2 \nu(\d q) =\int_{-\infty}^0 |Z(\theta)-\tilde{Z}(\theta)|^2 \nu(\d\theta-s)=0,\ \ s\ge 0.$$
So, by $\|\e^{A_2(t-s)}\|\le 1$ for $t\ge s$, Lemma \ref{L-PDE}, ${\bf(a4')}$ and the Fubini Theorem, when $\ll\ge\ll_0$ is large enough, we may find constants $C_1, C_2\in (0,\infty)$ such that
  \beg{equation*}\beg{split} |\LL(t)|^2 &
 \leq C_1\int_{0}^{t}\big\{|F(s,Z_{s})-F(s,\tilde{Z}_{s})|^{2}+|Z(s)-\tilde{Z}(s)|^2\big\} \d s\\
&\leq C_2 \int_{0}^{t}|Z(s)-\tilde{Z}(s)|^{2}\d s+C_2\int_{0}^{t}\d s\int_{-\infty}^{0}|Z(s+q)-\tilde{Z}(s+q)|^{2}\nu(\d q) \\
&=  C_2 \int_{0}^{t}|Z(s)-\tilde{Z}(s)|^{2}\d s+C_2\int_{0}^{t}\d s\int_{-s}^{0}|Z(s+q)-\tilde{Z}(s+q)|^{2}\nu(\d q) \\
&=  C_2 \int_{0}^{t}|Z(s)-\tilde{Z}(s)|^{2}\d s+C_2\int_{-t}^{0}\nu(\d q)\int_{-q}^t|Z(s+q)-\tilde{Z}(s+q)|^{2}\d s \\
&\le K(T) \int_0^t |Z(s)-\tilde{Z}(s)|^2\d s,\ \ t\in [0,T],
\end{split}\end{equation*} where $K(T):= C_2+C_2 \nu([-T,0))<\infty$ since $\nu$ is locally finite by \eqref{nu}.  Thus,
\beg{equation*}\beg{split} \int_0^r \e^{-2\ll t} \E\big[1_{\{t<\tau_m\}} |\LL(t)|^2\big]\,\d t&\le K(T) \E \int_0^r \e^{-2\ll t}1_{\{t<\tau_m\}}\d t \int_0^t |Z(s)-\tilde{Z}(s)|^2\d s\\
&\le K(T)\int_0^r \GG(t)\d t,\ \ r\in [0,T].\end{split}\end{equation*}
As a consequence, we have
\begin{align*}
&\int_0^r \e^{-2\ll s}  \E\big[1_{\{s<\tau_m\}} |Y(s)-\tilde{Y}(s)|^2\big]\,\d s\\
&\leq \int_0^r \e^{-2\ll t}  \E\Big\{1_{\{t<\tau_m\}}\Big(8|\LL(t)|^2+ \ff 8 7 |\Xi(t)|^2\Big)\Big\}\,\d t\\
&\leq \ff 6 7 \GG(r) +\ff 87 C_0\int_0^r\GG(t)\d t + 8 K(T) \int_0^r\GG(t)\d t, \ \ r\in[0,T].
\end{align*}
Since
\begin{equation*}\begin{split}
|Z(s)-\tilde{Z}(s)|^{2}&=|X(s)-\tilde{X}(s)|^2+|Y(s)-\tilde{Y}(s)|^2,
\end{split}\end{equation*}
and
\begin{equation}\label{X-Y1}\beg{split}
&|X(t\wedge\tau_{m})-\tilde{X}(t\wedge\tau_{m})|^{2}\leq C(T)\int_{0}^{t}|Y(s\wedge\tau_{m})-\tilde{Y}(s\wedge\tau_{m})|^{2}\d s,
 \quad t\in[0,T],
\end{split}\end{equation}
it follows from the Fubini theorem that
\begin{align*}
&\int_0^t \e^{-2\ll s}  \E\big[1_{\{s<\tau_m\}} |X(s)-\tilde{X}(s)|^2\big]\,\d s\\
&\leq C(T)\E \int_0^t\e^{-2\ll s} \d s\int_{0}^{s}|Z(r\wedge\tau_{m})-\tilde{Z}(r\wedge\tau_{m})|^{2}\d r\\
&\leq C(T)\int_0^t\GG(s)\d s,\quad t\in[0,T].
\end{align*}
Combining this with \eqref{X1}-\eqref{X3}, we arrive at
\beg{equation*}\beg{split} \GG(r)
&\le \ff 6 7 \GG(r) +\ff 87 C_0\int_0^r\GG(t)\d t + 8 K(T) \int_0^r\GG(t)\d t+C(T) \int_0^r\GG(t)\d t\\
&\le \ff 6 7 \GG(r) + 8(C_0+K(T)+C(T)) \int_0^r\GG(t)\d t,\ \ r\in [0,T].\end{split}\end{equation*}
Since by the definitions of $\GG$ and $\tau_m$ we have $\GG(t)<\infty$ for $t\in [0,T]$, it follows from   Gronwall's  inequality that  $\GG(T)=0$. Therefore,   \eqref{3.5} holds and the proof is finished.
\end{proof}

\subsection{Proof of Theorem \ref{T-EU-NE}}
\beg{proof}[Proof of Theorem \ref{T-EU-NE}]
(a) We first assume that ${\bf(a1)}$, ${\bf(a2)}$, ${\bf(a3')}$-${\bf(a4')}$ hold. Consider the following SPDE on $\mathbb{H}$:
\begin{equation*}
\begin{cases}
\d X^{\xi}(t)= \{A_1 X^{\xi}(t)+B Y^{\xi}(t)\}\d t, \\
\d Y^{\xi}(t)=A_2 Y^{\xi}(t)\d t+Q(t)\d W(t),\ \ (X^{\xi}(0),Y^{\xi}(0))=\xi(0).
\end{cases}
\end{equation*}
It is easy to see that the above equation has a uniqueness non-explosive mild solution:
\begin{equation*}
\begin{cases}
X^{\xi}(t)= \e^{A_1(t)} X^{\xi}(0)+\int_{0}^{t}\e^{A_1(t-s)}B Y^{\xi}(s)\d s, \\
Y^{\xi}(t)=\e^{A_2(t)} Y^{\xi}(0)+\int_{0}^{t}\e^{A_2(t-s)}Q(s)\d W(s),\ \  t\geq 0.
\end{cases}
\end{equation*}
Letting $(X^\xi_0, Y^\xi_0)=\xi$ and taking
\beg{equation*}\beg{split}
&W^{\xi}(t)=W(t)-\int_{0}^{t}\psi(s)\d s,\\
&\psi(s)= \big\{Q^{\ast}(QQ^{\ast})^{-1}\big\}(s) \big\{b(s,X^\xi(s),Y^\xi(s))+ F(s,X_s^\xi,Y_s^\xi)\big\},\ \ s,t\in[0,T],\end{split}
\end{equation*}
we have
\begin{equation*}
\begin{cases}
X^{\xi}(t)= \e^{A_1(t)} X^{\xi}(0)+\int_{0}^{t}\e^{A_1(t-s)}B Y^{\xi}(s)\d s, \\
Y^{\xi}(t)=\e^{A_2(t)} Y^{\xi}(0)+\int_{0}^{t}\e^{A_2(t-s)}\{b(s,X^{\xi}(s),Y^{\xi}(s))+F(s,X^{\xi}_s,Y^{\xi}_s)\}\d s\\
\qquad\qquad+\int_{0}^{t}\e^{A_2(t-s)}Q(s)\d W^{\xi}(s),\ \  t\in[0,T].
\end{cases}
\end{equation*}
Since $\|F\|_{T,\infty}+\|b\|_{T,\infty}<\infty$, Girsanov theorem implies that $\{W^{\xi}(t)\}_{t\in[0,T]}$ is a cylindrical Brownian motion on $\mathbb{H}_3$ under probability $\d\mathbb{Q}^{\xi}=R^{\xi}\d \mathbb{P}$, where
\beg{equation*}
R^{\xi}:=\exp\left[\int_{0}^{T}\big\langle \psi(s),\d W(s)\big\rangle-\frac{1}{2}\int_{0}^{T}\big|\psi(s)\big|^{2}\d s\right].
\end{equation*}
Then, under the probability $\Q^\xi$,  $((X^{\xi}(t),Y^{\xi}(t)),W^{\xi}(t))_{t\in[0,T]}$ is a weak mild solution to \eqref{IE1}. On the other hand, by Proposition \ref{P-uni},   the pathwise uniqueness holds for the mild solution to \eqref{IE1}. So, by the Yamada-Watanabe principle, see \cite{YW}, the equation \eqref{IE1} has a unique mild solution. Moreover, in this case the solution is non-explosive.

(b) In general, take $\psi\in C_{b}^{\infty}([0,\infty))$ such that $0\leq \psi\leq 1$, $\psi(u)=1$ for $u\in[0,1]$ and $\psi(u)=0$ for $u\in[2,\infty)$. For any $m\geq1$, let
\beg{equation*}\begin{split}
&b^{[m]}(t,z)=b(t\wedge m,z)\psi(|z|/m),\ \ (t,z)\in[0,\infty)\times\mathbb{H},\\
&F^{[m]}(t,\xi)=F(t\wedge m,\xi)\psi(\|\xi\|_{\C_\nu}/m),\ \ (t,\xi)\in[0,\infty)\times\C_\nu.
\end{split}\end{equation*}
By {\bf (a3)}-{\bf (a4)} and the local boundedness of $F$, we know $F^{[m]}$ and $b^{[m]}$ satisfy ${\bf (a3^{'})}-{\bf (a4^{'})}$. Then by (a), \eqref{IE1} for $F^{[m]}$ and $b^{[m]}$ in place of $F$, $b$ has a unique mild solution $Z^{[m]}(t):=(X^{[m]}(t),Y^{[m]}(t))$ starting at $(X_{0},Y_{0})$ which is non-explosive. Let
\beg{equation*}
\zeta_{0}=0,\ \ \zeta_{m}=m\wedge\inf\{t\geq 0:|(X^{[m]}(t),Y^{[m]}(t))|\geq m\},\ \ m\geq 1.
\end{equation*}
Since $F^{[m]}(s,\xi)=F(s,\xi)$ and $b^{[m]}(s,\xi(0))=b(s,\xi(0))$ hold for $s\leq m$, and $\|\xi\|_{\C_\nu}\leq m$,  by Proposition \ref{P-uni}, for any $n$, $m\geq1$, we have $Z^{[m]}(t)=Z^{[n]}(t)$ for $t\in[0,\zeta_{m}\wedge\zeta_{n}]$. In particular, $\zeta_{m}$ is increasing in $m$. Let $\zeta=\lim_{m\to\infty}\zeta_{m}$ and
\beg{equation*}
Z(t)=\sum_{m=1}^{\infty}1_{[\zeta_{m-1},\zeta_{m})}Z^{[m]}(t),\ \ t\in[0,\zeta).
\end{equation*}
Then it is easy to see that $Z(t)_{t\in[0,\zeta)}$ is a mild solution to \eqref{IE1} with lifetime $\zeta$ and, due to Proposition \ref{P-uni}, the mild solution is unique. So we end the proof of Theorem \ref{T-EU-NE} (1).

(c) Next, we prove the non-explosion.

Let $\Phi, h$ satisfy \eqref{Bb-phi}. Let $(Z(t))_{t\in(-\infty,\zeta)}$ be the mild solution to \eqref{IE1} with lifetime $\zeta$. Set $M(t)=\int_{0}^{t} \e^{A_2(t-s)}Q(s)\d W(s)$, $t\in[0,\zeta)$; $M(t)=0$, $t\in(-\infty,0]$.
It is clear that $(X(t),\tilde{Y}(t)):=(X(t), Y(t)-M(t))$ is the mild solution to the following equation up to $\zeta$,
\beg{equation*}
\begin{cases}
\d X(t)= \{A_1 X(t)+B (\tilde{Y}(t)+M(t))\}\d t, \\
\d \tilde{Y}(t)=\{A_2 \tilde{Y}(t)+b(t,X(t),\tilde{Y}(t)+M(t))+F(t,X_t,\tilde{Y}_t+M_t)\}\d t,
\end{cases}
\end{equation*}
Since $A_2$ is negative definite, then \eqref{Bb-phi} and Ito's formula imply that for any $T>0$,
\beg{equation}\label{Y-phi}\beg{split}
\d |\tilde{Y}(t)|^{2}&\leq 2\langle \tilde{Y}(t), b(t,X(t),\tilde{Y}(t)+M(t))+F(t,X_t,\tilde{Y}_{t}+M_{t})\rangle\d t\\
&\leq2\left(\Phi_{\zeta\wedge T}(\|X_{t}\|_{\C^1_\nu}^{2}+\|\tilde{Y}_{t}\|_{\C^2_\nu}^{2})+h_{\zeta\wedge T}(\|M_{t}\|_{\C^2_\nu})\right)\d t,\ \ t\in[0,\zeta\wedge T).
\end{split}\end{equation}
This yields that
\beg{equation}\label{XX}\beg{split}
|\tilde{Y}(t)|^{2}&\leq |Y(0)|^{2}+2\int_{0}^{t}h_{T}(\|M_{s}\|_{\C^2_\nu})\d s\\
&+2\int_{0}^{t}\Phi_{T}(\|X_{s}\|_{\C^1_\nu}^{2}+\|\tilde{Y}_{s}\|_{\C^2_\nu}^{2})\d s, \  \ t\in [0, T\land \zeta).
\end{split}\end{equation}
 Since $\tilde{Y}_0=Y_0$,    \eqref{nu} implies
\beg{equation*}\beg{split} \|\tilde{Y}_s\|_{\C^2_\nu}^2 & =  |\tilde{Y}(s)|^2 + \int_{-\infty}^{-s} |\tilde{Y}(s+\theta)|^2\nu(\d \theta) + \int_{-s}^{0} |\tilde{Y}(s+\theta)|^2\nu(\d \theta)\\
&\le \big\{1+\nu([-s,0))\big\} \sup_{r\in [0,s]} |\tilde{Y}(r)|^2 +\kk(s) \int_{-\infty}^0 |Y_0(\theta)|^2 \nu(\d\theta)\\
&\le \kk(T) \|Y_0\|_{\C^2_\nu}^2 + \big\{1+\nu([-T,0)) \big\}\sup_{r\in [0,s]}|\tilde{Y}(r)|^2,\ \ s\in [0, T\land \zeta).\end{split}\end{equation*}
Similarly, we have
\begin{equation}\label{x-in}\begin{split}
\|X_s\|^2_{\C^1_\nu}&\le \kk(T) \|X_0\|_{\C^1_\nu}^2 + \big\{1+\nu([-T,0)) \big\}\sup_{r\in [0,s]}|X(r)|^2,\ \ s\in[0, T\land \zeta).
\end{split}\end{equation}
On the other hand, since $(X,Y)$ is the solution to \eqref{IE1}, it is clear that
\begin{equation*}\begin{split}
\sup_{r\in [0,s]}|X(r)|^2&\leq \|X_0\|^2_{\C^1_\nu}+\sup_{t\in[0,s]}\left|\int_0^t\e^{A_1(t-v)}B (\tilde{Y}(v)+M(v))\d v\right|^2,\ \ s\in[0, T\land \zeta).
\end{split}\end{equation*}
Combining this with \eqref{x-in}, there exists a random variable $N>1$ such that
\begin{equation}\label{x-in'}\begin{split}
\|X_s\|^2_{\C^1_\nu}&\le  N+(N-1)\sup_{r\in [0,s]}|\tilde{Y}(r)|^2,\ \ s\in[0, T\land \zeta).
\end{split}\end{equation}
So, by letting
\beg{equation*}\beg{split}
\aa(T)= |Y(0)|^{2}+2\int_{0}^{T}h_{T}(\|M_{s}\|_{\C^2_\nu})\d s,\end{split}\end{equation*}
\eqref{XX} implies
\beg{equation}\label{sup-Y}
\sup_{r\in [0,s]}|\tilde{Y}(r)|^2\leq \alpha(T)+2\int_{0}^{s}\Phi_{T}\left(N\sup_{r\in [0,q]}|\tilde{Y}(r)|^2+N\right)\d q, \ \ s\in[0,\zeta\wedge T)
\end{equation}
for some random variable $N$.
Let $$\Psi_{T}(s)=\int_{1}^{s}\frac{\d v}{2\Phi_{T}(N+Nv)}.$$
By Bihari-LaSalle inequality, \eqref{sup-Y} implies
\beq\label{Z-psi}
\sup_{r\in [0,t]}|\tilde{Y}(r)|^2\leq \Psi_{T}^{-1}(\Psi_{T}(\alpha(T))+t),\ \ t\in[0,\zeta\wedge T).
\end{equation}
Moreover, since $M(t)$ is continuous, then $\sup_{t\in[0, T)}|M(t)|^{2}<\infty$. Thus by the definition of $\zeta$ and $\tilde{Y}$, on the set $\{\zeta<\infty\}$, we have $\mathbb{P}$-a.s.
\beq\label{lim-YX}
\limsup_{t\uparrow\zeta}(|\tilde{Y}(t)|^2+|X(t)|^2)=\limsup_{t\uparrow\zeta}(|X(t)|^2+|Y(t)|^2)=\infty.
\end{equation}
Moreover on the set $\zeta\leq T$, $\mathbb{P}$-a.s. $\alpha(T)<\infty$. Combining the property of $\Phi$, \eqref{x-in'}, \eqref{Z-psi} and \eqref{lim-YX}, it holds that on the set $\zeta\leq T$, $\mathbb{P}$-a.s.
\beg{equation*}
\limsup_{t\uparrow\zeta}(|\tilde{Y}(t)|^2+|X(t)|^2)\leq N\Psi_{T}^{-1}(\Psi_{T}(\alpha(T))+T)+N<\infty.
\end{equation*}
So for any $T>0$, $\mathbb{P}\{\zeta\leq T\}=0$. Note that
\beg{equation*}
 \mathbb{P}\{\zeta< \infty\}=\mathbb{P}\left(\bigcup_{m=1}^{\infty}\{\zeta\leq m\}\right)\leq\sum_{m=1}^{\infty}\mathbb{P}\{\zeta\leq m\}=0,
 \end{equation*}
and this implies the solution of \eqref{IE1} is non-explosive.
\end{proof}
\section{Harnack inequalities}
Throughout of this section, we assume the length of time delay is finite. Since the log-Harnack inequality implies the strong Feller property (see \cite[Theorem 1.4.1]{Wbook}), and it is easy to
 see that $P_T$ is strong Feller only if supp$\,\nu\subset [-T,0]$, we see that the restriction on bounded time delay is essential for the study, see also \cite{WY}.
 Let $r\in (0,\infty)$ such that supp$\,\nu\subset [-r,0]$. In this case $\C_\nu$ is reformulated as
$$\C_\nu=\bigg\{\xi: [-r,0]\to \mathbb{H} \ \text{is\ measurable\ with}\ \nu(|\xi|^2):=\int_{-r}^0 |\xi(\theta)|^2\nu(\d\theta)<\infty\bigg\}.$$
In this section, we consider the following functional SPDEs:
\beq\label{EH}
\begin{cases}
\d X(t)= \{A_1 X(t)+B Y(t)\}\d t, \\
\d Y(t)=\{A_2 Y(t)+b(X(t),Y(t))+F(X_t,Y_t)\}\d t+Q\d W(t),
\end{cases}
\end{equation}
where $B\in\L\left(\mathbb{H}_2; \mathbb{H}_1\right)$, for any $i=1,2$, $(A_i,\D(A_i))$ is a bounded above linear operator generating a strongly continuous semigroup $\e^{tA_i}$ on $\mathbb{H}_i$, $F: \C_{\nu}\to \mathbb{H}_2$, $b: \mathbb{H}\to \mathbb{H}_2$ and $Q\in \L\left(\mathbb{H}_3; \mathbb{H}_2\right)$.
We make the following assumptions:

\beg{enumerate}
\item[{\bf (A2)}]
\beg{enumerate}
\item[(i)]  $Q Q^{\ast}$ is invertible.

\item[(ii)]  $BB^\ast$ is invertible in $\mathbb{H}_1$, and $B\e^{A_2 t} = \e^{A_1 t} \e^{A_0 t} B$ for some $A_0 \in \L(\mathbb{H}_1,\mathbb{H}_1)$ and all $t\geq 0$.

\item[(iii)] There exists $n_0 \geq1$ such that for any $n \geq n_0$, $\pi^{(n)}_1 B = B\pi^{(n)}_2$ on $\mathbb{H}_2$, and $\pi^{(n)}_1 A_1 =
A_1\pi^{(n)}_1$ on $\D(A_1)$.
\end{enumerate}
\item[$\bf{(A3)}$] $\|b\|_{\infty}<\infty$. There exists $\phi\in\D$ and constants $K>0$, $\alpha\in(\frac{2}{3},1)$ such that

and for any $ x,x'\in\mathbb{H}_1, y,y'\in\mathbb{H}_2$,
\beq\label{b-phi0'}
|b(x,y)-b(x',y')|\leq K|x-x'|^{\alpha}+\phi(|y-y'|).
\end{equation}
\item[$\bf{(A4)}$] There exists a constant $c>0$ such that
\beg{equation*}\beg{split}
|F(\xi)-F(\eta)|\leq c\|\xi-\eta\|_{\C_\nu},\ \ \xi,\eta\in\C_\nu.
\end{split}\end{equation*}

\item[$\bf{(A5)}$] Let $A_0$ be in {\bf(A2)} (ii). Assume $$\Lambda_t:= \int_0^t\e^{sA_0} BB^*\e^{s A_0^*}\d s,\ \ t\geq 0$$
is invertible on $\mathbb{H}_1$. Moreover, $A_1\leq \delta-\lambda_1$ for some constant $\delta>0$; i.e., $\<A_1 x,x\>\leq (\delta-\lambda_1)|x|^2$ for all $x \in\D(A_1)$, where $\lambda_1$ is in {\bf (a1)}.
\end{enumerate}

\begin{rem}\label{non-ex} Obviously, when $\mathbb{H}_1 = \mathbb{H}_2=\mathbb{H}_3$, $Q= B = I$, $A_1 = A_2$ and {\bf(a1)} holds, then {\bf(A2)} and {\bf(A5)} hold.
See \cite{W2} for more examples, where $\mathbb{H}_2$ might be a subspace of $\mathbb{H}_1$.
\end{rem}
 By Theorem \ref{T-EU-NE}, {\bf(a1)}, {\bf(A2)}-{\bf(A4)} imply that \eqref{EH} has a unique non-explosive mild solution. {\bf{(A5)}} will be used to prove the Harnack and shift Harnack inequalities in the next section. Let
$(X^{(\xi_1,\xi_2)}(t), Y^{(\xi_1,\xi_2)}(t))$ be the unique continuous mild solution for initial point $(\xi_1,\xi_2)\in\C_\nu$. Let $P_t$ be the Markov semigroup generated
by the segment solution, i.e.
$$P_t f(\xi_1,\xi_2)=\mathbb{E} f(X^{(\xi_1,\xi_2)}_t, Y^{(\xi_1,\xi_2)}_t),\ \ t\geq 0, f\in\B_b(\C_\nu).$$

We will use the coupling constructed in \cite[Lemma 4.1]{BWY15} to derive the Harnack inequalities.
\begin{thm}\label{THar} Assume {\bf (a1)},
  {\bf (A2)}-{\bf (A5)} and let $T>r$. Then for any $\xi=(\xi_1, \xi_2), h=(h_1, h_2)\in \C_\nu$ and positive
$f\in \B_b(\C_\nu)$,
\beg{equation*}\beg{split}
&P_T\log f(\xi+h)\leq\log P_T f(\xi)+ \Sigma(T,h,r)
,\end{split}\end{equation*}
and
 \beg{equation*}\beg{split}  (P_Tf)^p(\xi+h)\le  &P_Tf^p(\xi)
  \exp\bigg[\ff{p}{2(p-1)} \Sigma(T,h,r)\bigg],
  \end{split}\end{equation*}
where
\begin{align*}
\Sigma(h,T,r)&=C(T-r)\Big(\ff 1 {T-r}|h_2(0)|+ \|B\||h(0)|\Big)^2\\
&+CT(|h_1(0)|+ \|B\||h(0)|)^{2\alpha}\\
&+CT\phi^2\left(C(|h_2(0)|+ \|B\||h(0)|)\right)\\
&+CT\Big(\|h\|_{\C_\nu}+\|B\||h(0)|\Big)^2,
\end{align*}
and $C>0$ is a constant. In addition, since $\lim_{\|h\|_{\C_\nu}\to0}\Sigma(h,T,r)=0$, $P_T$ is strong Feller for any $T>r$. 
 \end{thm}
\begin{proof}  Fix $T>r$ and $\xi=(\xi_1,\xi_2), h=(h_1,h_2)\in\C_\nu$. For any $\eta\in\C_\nu$, let $(X^\eta(t), Y^\eta(t))$ solve (\ref{EH})  with $(X_0,Y_0)= \eta$. For simplicity, we set $(X(t),Y(t))=(X^\xi(t), Y^\xi(t))$.
Let $$\gamma(t)=t(T-r-t)^{+}B^\ast\e^{A^\ast_0t}e,$$
and

\begin{align*}
e=-\bar{\Lambda}_{T-r}^{-1} \bigg(h_1(0)+\int_0^{T-r} \ff{T-r-u}{T-r} \e^{A_0u}Bh_2(0)\d u\bigg),
\end{align*}
where
$$\bar{\Lambda}_{T-r}:= \int_0^{T-r} s(T-r-s)\e^{A_0s} BB^*\e^{s A^*_0s}\d s.$$
Let $(\bar{X}(t), \bar{Y}(t))$ solve the equation
\beq\label{EC1} \beg{cases} &\d\bar{X}(t)= \{A_1\bar{X}(t)+B\bar{Y}(t)\}\d t,\\
&\d \bar{Y}(t)= \{A_2\bar{Y}(t)+b(X(t), Y(t))+F(X_t,Y_t)\}\d t +Q \d W(t)\\
&\ \ \ \ \ \ \ \ \ \ \ \ +\e^{A_2t} \left\{-\frac{1_{[0,T-r]}(t)}{T-r}h_2(0)+\gamma'(t)\right\}\d t\end{cases}\end{equation} with
$(\bar{X_0},\bar{Y_0})= \xi+h$. Moreover, let
 $$\Gamma_2(t):= \beg{cases} h_2(t),\ \  &\text{if}\  t\in[-r,0],\\
\e^{A_2t} \left\{\frac{(T-r-t)^+}{T-r}h_2(0)+\gamma(t)\right\}, &\text{if}\
t\in(0,T].\end{cases} $$
and
 $$\Gamma_1(t):= \beg{cases} h_1(t),\ \  &\text{if}\  t\in[-r,0],\\
\e^{tA_1}h_1(0)+\int_0^{t} \e^{(t-u)A_1} B\Gamma_2(u)\d u\}, &\text{if}\
t\in(0,T].\end{cases} $$
Then it is not difficult to see that
\beq\label{EE} (\bar{X}(s),\bar{Y}(s))=(X(s),Y(s))+\Gamma(s), \ \
s\in[-r,T]\end{equation} holds for
 $$\Gamma(s)=(\Gamma_1(s), \Gamma_2(s)), \ \ s\in[-r,T].$$
 In particular, by the definition of $\gamma$,
 we have $\Gamma_T=0$, which implies $$(\bar{X}_T, \bar{Y}_T)=(X_T,Y_T).$$
 Thus
 let
\begin{align*}
\Phi(t)&=b(X(t), Y(t))-b(\bar{X}(t),
\bar{Y}(t))+F(X_t,Y_t)-F(\bar{X}_t,\bar{Y}_t)\\
&+\e^{A_2t} \left\{-\frac{1_{[0,T-r]}(t)}{T-r}h_2(0)+\gamma'(t)\right\}.
\end{align*}
Set
\begin{align*}
R(s)=\exp\bigg[-\int_0^s\< (QQ^\ast)^{-1}\Phi(u), \d
W(u)\>-\frac{1}{2}\int_0^s |(QQ^\ast)^{-1}\Phi(u)|^2\d u\bigg],
\end{align*}
and
$$
\bar{W}(s)=W(s)+\int_0^s(QQ^\ast)^{-1}\Phi(u)\d u.
$$
Then (\ref{EC1}) reduces to
\beq\label{E2'}   \beg{cases} \d \Bar{X}(t)= \{A_1\Bar{X}(t)+B\Bar{Y}(t)\}\d t,\\
\d \Bar{Y}(t)= \{A_2\bar{Y}(t)+b(\Bar{X}(t), \Bar{Y}(t))+F(\bar{X}_t,\bar{Y}_t)\}\d t +Q \d \bar{W}(t).
\end{cases}
\end{equation}
Thus the distribution of $(\bar{X}_T,\bar{Y}_T)$ under $\Q_T=R(T)\P$ coincides with the one of $(X^{\xi+h}_T,Y^{\xi+h}_T)$  under $\P$.
In addition, by the definition of $\gamma(t)$, there exists a constant $C>0$ such that for any $t\in[0,T]$,
\beg{equation}\label{NN0}\beg{split}
&\left|\e^{A_2t} \left\{-\frac{1_{[0,T-r]}(t)}{T-r}h_2(0)+\gamma'(t)\right\}\right|\le C1_{[0,T-r]}(t)\Big(\ff 1 {T-r}|h_2(0)|+ \|B\||h(0)|\Big),\\
&|\Gamma_1(t)|\le C(|h_1(0)|+ \|B\||h(0)|),\\
&|\Gamma_2(t)|\le C(|h_2(0)|+ \|B\||h(0)|),\end{split}\end{equation}
and \eqref{nu} yields that
\beg{equation}\label{NN01}\beg{split}
&\|\Gamma_t\|_{\C_\nu} \le C\Big(\|h\|_{\C_\nu}+\|B\||h(0)|\Big).\end{split}\end{equation}
Thus, from {\bf(A3)}-{\bf(A4)}, there exists a constant $C>0$ such that
\begin{equation}\begin{split}\label{Phi'}
&\int_0^T|\Phi(s)|^2\d s\\
&\leq C\int_0^T\left(|\Gamma_1(s)|^{\alpha}+\phi(|\Gamma_2(s)|)+\|\Gamma_s\|_{\C_\nu}+\left|-\frac{1_{[0,T-r]}(s)}{T-r}h_2(0)+\gamma'(s)\right|\right)^2\d s\\
&\leq CT(|h_1(0)|+ \|B\||h(0)|)^{2\alpha}\\
&+CT\phi^2\left(C(|h_2(0)|+ \|B\||h(0)|)\right)\\
&+CT\Big(\|h\|_{\C_\nu}+\|B\||h(0)|\Big)^2\\
&+C(T-r)\Big(\ff 1 {T-r}|h_2(0)|+ \|B\||h(0)|\Big)^2.
\end{split}\end{equation}

On the other hand, by Young's inequality,
\begin{align*}
P_T \log f(\xi+h)&=\E^{\mathbb{Q}_T}\log f(\bar{X}_T,\bar{Y}_T)\\
&=\E ^{\Q_T}\log f(X^\xi_T,Y^\xi_T)\leq \log P_T f(\xi)+\E R(T)\log R(T),
\end{align*}
and by H\"{o}lder inequality,
\begin{align*}
P_T f(\xi+h)&=\E^{\mathbb{Q}_T}f(\bar{X}_T,\bar{Y}_T)\\
&=\E ^{\Q_T}f(X^\xi_T,Y^\xi_T)\leq (P_T f^p(\xi))^{\frac{1}{p}}\{\E R(T)^{\frac{p}{p-1}}\}^{\frac{p-1}{p}}.
\end{align*}
Since $\bar W$ is a cylindrical Brownian motion under $\Q_T$, by the definition of $R(T)$, it is easy to see that
\begin{align*}
\E R(T)\log R(T)=\E ^{\Q_T}\log R(T)=\frac{1}{2}\E^{\Q_T} \int_0^T |(QQ^\ast)^{-1}\Phi(u)|^2\d u,
\end{align*}
and
\begin{align*}
&\E R(T)^{\frac{p}{p-1}}\\
&\leq \E\Bigg\{\exp\bigg[-\frac{p}{p-1}\int_0^T\< (QQ^\ast)^{-1}\Phi(u), \d
W(u)\>-\frac{1}{2}\frac{p^2}{(p-1)^2}\int_0^s |(QQ^\ast)^{-1}\Phi(u)|^2\d u\bigg]\\
&\ \ \ \ \ \ \ \ \ \times\exp\bigg[\frac{1}{2}\frac{p^2}{(p-1)^2}\int_0^s |(QQ^\ast)^{-1}\Phi(u)|^2\d u-\frac{1}{2}\frac{p}{p-1}\int_0^s |(QQ^\ast)^{-1}\Phi(u)|^2\d u\bigg]\Bigg\}\\
&\leq\mathrm{ess}\sup_{\Omega}\exp\left\{\frac{p}{2(p-1)^2}\int_0^T |(QQ^\ast)^{-1}\Phi(u)|^2\d u\right\}.
\end{align*}
Combining this with {\bf(A2)} (i) and \eqref{Phi'}, we derive the Harnack inequalities. Finally, the strong Feller of $P_T$ for $T>r$ follows from \cite[Theorem 1.4.1 (1)]{Wbook} since $\lim_{\|h\|_{\C_\nu}\to0}\Sigma(h,T,r)=0$.
\end{proof}
The following corollary is a direct conclusion of Theorem \ref{THar}, see \cite[Theorem 1.4.2]{Wbook}.
\begin{cor}\label{density} Let the assumption of Theorem \ref{THar} hold. For any $T>r$, $\xi,\eta\in\C_\nu$, let $\Sigma(T,\eta-\xi,r)$ is defined in Theorem \ref{THar} for $\eta-\xi$ instead of $h$.  Then $P_T(\xi,\cdot)$ is equivalent to $P_T(\eta,\cdot)$. Moreover,  for any $p>1$,
$$P_T\left\{\left(\frac{\d P_T(\xi,\cdot)}{\d P_T(\eta,\cdot)}\right)^{\frac{1}{p-1}}\right\}(\xi)\leq \exp\bigg[\ff{p}{2(p-1)^2} \Sigma(T,\eta-\xi,r)\bigg],$$
and
$$P_T\left\{\log\frac{\d P_T(\xi,\cdot)}{\d P_T(\eta,\cdot)}\right\}(\xi)\leq  \Sigma(T,\eta-\xi,r).$$
In addition, if $\mu$ is an invariant probability measure of $P_t$, then the entropy-cost inequality
$$\mu((P_T^\ast f)\log P_T^\ast f )\leq \mathbb{W}_1^{\Sigma}(f\mu,\mu),\ \  f\geq 0, \mu(f)=1$$
holds for $P_T^\ast$, the adjoint operator of $P_T$ in $L^2(\mu)$, where $\mathbb{W}_1^{\Sigma}$ is the $L^1$ transportation cost induced by the cost function $\Sigma(T,\cdot,r)$, i.e. for any two probability $\mu_1,\mu_2$ on $\C_\nu$, $$ \mathbb{W}_1^{\Sigma}(\mu_1,\mu_2):=\inf_{\pi\in \mathbf{C}(\mu_1,\mu_2)}\int_{\C_\nu\times\C_\nu}\Sigma(T,\eta-\xi,r)\pi(\d\xi,\d \eta),$$
where $\mathbf{C}(\mu_1,\mu_2)$ is the set of all couplings of $\mu_1,\mu_2$.
\end{cor}
\section{Shift Harnack inequalities}
In this section, due to some technique difficulty for the construction of coupling by change of measure in the finite dimension, we assume supp $\nu$ is $\{0\}$, i.e. the case without delay. In other words, $\C_\nu=\mathbb{H}$ and we consider the following SPDEs:
\beq\label{ESH}
\begin{cases}
\d X(t)= \{A_1 X(t)+B Y(t)\}\d t, \\
\d Y(t)=\{A_2 Y(t)+b(X(t),Y(t))+F(X(t),Y(t))\}\d t+Q\d W(t),
\end{cases}
\end{equation}
The main result on the shift Harnack inequality for $P_T, T>0$ is the following theorem.
\begin{thm}\label{TsHar} Assume $\mathrm{supp}$ $\nu$ is $\{0\}$. For any $\eta=(\eta_1, \eta_2)\in \mathbb{H}$, let
  $$\eta(t)=(\eta_1(t),\eta_2(t))=\left(\int_0^t\e^{sA_1}\eta_1\d s,\int_0^t\e^{sA_2}\eta_2\d s\right),\ \ t\geq 0.$$
Assume {\bf(a1)},
  {\bf (A2)}-{\bf (A5)} and let $T>0$. Then for any $\xi=(\xi_1, \xi_2)\in \mathbb{H}$ and any positive $f\in \B_b(\mathbb{H})$,
\beg{equation*}\beg{split}
P_T\log f(\xi)\leq\log (P_T f(\eta(T)+\cdot))(\xi)+ \beta(T,\eta),\end{split}\end{equation*}
and
 \beg{equation*}\beg{split}  (P_Tf)^p(\xi)\le  &P_T(f^p(\eta(T)+\cdot))(\xi)
  \exp\bigg[\ff{p}{2(p-1)} \beta(T,\eta)\bigg],\end{split}\end{equation*}
where
\begin{align*}
\beta(T,\eta)&= CT\|B\|^{2\alpha}(T^2|\eta_2|+ T^3\|B\||\eta|)^{2\alpha}\\
&+CT\phi^2\left( C(T|\eta_2|+ T^2\|B\||\eta|)\right)\\
&+CT\left[\left( T|\eta_2|+ T^2\|B\||\eta|\right)^2+\left(\|B\|(T^2|\eta_2|+ T^3\|B\||\eta|)\right)^2\right]\\
&+CT\Big(|\eta _2|+ \|B\||\eta|\Big)^2.
\end{align*}
and $C>0$ is a constant.
 \end{thm}
\begin{proof} Fix $T>0$ and $\xi=(\xi_1,\xi_2), \eta=(\eta_1,\eta_2)\in\mathbb{H}$. For any $h\in\mathbb{H}$, let $(X^h(t), Y^h(t))$ solve (\ref{ESH})  with $(X(0),Y(0))= h$. For simplicity, we set $(X(t),Y(t))=(X^\xi(t), Y^\xi(t))$.
Let $$\tilde{\gamma}(t)=t(T-t)B^\ast\e^{A^\ast_0t}\tilde{e},\ \ t\in[0,T],$$
where
\begin{align*}
\tilde{e}=\tilde{\Lambda}_{T}^{-1} \bigg(\int_0^T\e^{-A_1s}\eta_1\d s-\int_0^{T}\e^{-A_1u}B\eta_2(u)\d u\bigg),
\end{align*}
and
$$\tilde{\Lambda}_{T}:= \int_0^{T} s(T-s)\e^{A_0s} BB^*\e^{A^*_0s}\d s.$$
Let $(\tilde {X}(t), \tilde{Y}(t))$ solve the equation
\beq\label{EC1s} \beg{cases} &\d \tilde{X}(t)= \{A_1\tilde{X}(t)+B\tilde{Y}(t)\}\d t,\\
&\d \tilde{Y}(t)= \{A_2\tilde{Y}(t)+b(X(t), Y(t))+F(X(t),Y(t))\}\d t +Q \d W(t)\\
&\ \ \ \ \ \ \ \ \ \ \ \ + \left\{\eta_2+\e^{A_2 t}\tilde{\gamma}'(t)\right\}\d t\end{cases}\end{equation} with
$(\tilde{X}(0),\tilde{Y}(0))= \xi$.
Moreover, let
 $$\tilde{\Gamma}_2(t)= \eta_2(t)+\e^{A_2t}\tilde{\gamma}(t),\ \ \tilde{ \Gamma}_1(t)=
\int_0^{t} \e^{(t-u)A_1} B\tilde{\Gamma}_2(u)\d u. $$
Then it is not difficult to see that
\beq\label{EEs} (\tilde{X}(s),\tilde{Y}(s))=(X(s),Y(s))+\tilde{\Gamma}(s), \ \
s\in[0,T]\end{equation} holds for
 $$\tilde{\Gamma}(s)=(\tilde{\Gamma}_1(s), \tilde{\Gamma}_2(s)), \ \ s\in[0,T].$$
 In particular, by the definition of $\tilde{\gamma}$,
 we have $\tilde{\Gamma}(T)=(\eta_1(T),\eta_2(T))$, which implies $$(\tilde{X}(T), \tilde{Y}(T))=(X(T)+\eta_1(T),Y(T)+\eta_2(T)).$$
 Thus
 let
\begin{align*}
\tilde{\Phi}(t)&=b(X(t), Y(t))-b(\tilde{X}(t),
\tilde{Y}(t))\\
&+F(X(t),Y(t))-F(\tilde{X}(t),\tilde{Y}(t))+\left\{\eta_2+\e^{A_2 t}\tilde{\gamma}'(t)\right\}.
\end{align*}
Set
\begin{align*}
\tilde{R}(s)=\exp\bigg[-\int_0^s\< (QQ^\ast)^{-1}\tilde{\Phi}(u), \d
W(u)\>-\frac{1}{2}\int_0^s |(QQ^\ast)^{-1}\tilde{\Phi}(u)|^2\d u\bigg],
\end{align*}
and
$$
\tilde{W}(s)=W(s)+\int_0^s(QQ^\ast)^{-1}\tilde{\Phi}(u)\d u.
$$
Then (\ref{EC1s}) reduces to
\beq\label{E2s}   \beg{cases} \d \tilde X(t)= \{A_1\tilde X(t)+B\tilde Y(t)\}\d t,\\
\d \tilde Y(t)= \{A_2\tilde Y(t)+b(\tilde X(t), \tilde Y(t))+F(\tilde X(t),\tilde Y(t))\}\d t +Q \d \tilde{W}(t).
\end{cases}
\end{equation}
Thus the distribution of $(\tilde X(T),\tilde Y(T))$ under the new probability $\tilde{\Q}_T=\tilde{R}(T)\P$ coincides with the one of $(X(T),Y(T))$  under $\P$.
Moreover, there exists a constant
$C>0$ such that for any $s\in[0,T]$,
\beg{equation}\label{NN0s}\beg{split}
&\left|\eta_2+\e^{A_2 t}\tilde{\gamma}'(t)\right|\le C\Big(|\eta _2|+ \|B\||\eta|\Big),\\
&|\tilde{\Gamma}_1(s)|\le C\|B\|(T^2|\eta_2|+ T^3\|B\||\eta|),\\
&|\tilde{\Gamma}_2(s)|\le C(T|\eta_2|+ T^2\|B\||\eta|)\\
&|\tilde{\Gamma}(s)|\le C(T|\eta_2|+ T^2\|B\||\eta|)+C\|B\|(T^2|\eta_2|+ T^3\|B\||\eta|).\end{split}\end{equation}
Thus, from {\bf(A3)}-{\bf(A4)}, it holds
\begin{equation}\begin{split}\label{Phis}
\int_0^T|\tilde{\Phi}(s)|^2\d s&\leq C\int_0^T\left(|\tilde{\Gamma}_1(s)|^{\alpha}+\phi(|\tilde{\Gamma}_2(s)|)+|\tilde{\Gamma}(s)|+\left|\eta_2+\e^{A_2 s}\tilde{\gamma}'(s)\right|\right)^2\d s\\
&\leq CT\|B\|^{2\alpha}(T^2|\eta_2|+ T^3\|B\||\eta|)^{2\alpha}\\
&+CT\phi^2\left( C(T|\eta_2|+ T^2\|B\||\eta|)\right)\\
&+CT\left[\left( T|\eta_2|+ T^2\|B\||\eta|\right)^2+\left(\|B\|(T^2|\eta_2|+ T^3\|B\||\eta|)\right)^2\right]\\
&+CT\Big(|\eta _2|+ \|B\||\eta|\Big)^2.
\end{split}\end{equation}
On the other hand, by Young's inequality,
\begin{align*}
P_T \log f(\xi)&=\E^{\tilde{\Q}_T}\log f(\tilde{X}(T),\tilde{Y}(T))\\
&=\E ^{\tilde{\Q}_T}\log f(X(T)+\eta_1(T),Y(T)+\eta_2(T))\\
&\leq \log P_T f(\cdot+\eta)(\xi)+\E \tilde{R}(T)\log \tilde{R}(T),
\end{align*}
and by H\"{o}lder inequality,
\begin{align*}
P_T f(\xi)&=\E^{\tilde{\Q}_T}f(\tilde{X}(T),\tilde{Y}(T))\\
&=\E ^{\tilde{\Q}_T}f(X(T)+\eta_1(T),Y(T)+\eta_2(T))\leq (P_T f^p(\cdot+\eta))^{\frac{1}{p}}(\xi)\{\E \tilde{R}(T)^{\frac{p}{p-1}}\}^{\frac{p-1}{p}}.
\end{align*}
Similarly to the estimate of $R(T)$ in section 3, it is easy to see that
\begin{align*}
\E \tilde{R}(T)\log \tilde{R}(T)=\E ^{\tilde{\Q}_T}\log \tilde{R}(T)=\frac{1}{2}\E^{\tilde{\Q}_T} \int_0^T |(QQ^\ast)^{-1}\tilde{\Phi}(u)|^2\d u.
\end{align*}
and
\begin{align*}
\E \tilde{R}(T)^{\frac{p}{p-1}}\leq\mathrm{ess}\sup_{\Omega}\exp\left\{\frac{p}{2(p-1)^2}\int_0^T |(QQ^\ast)^{-1}\tilde{\Phi}(u)|^2\d u\right\}.
\end{align*}
Thus the shift Harnack inequality follows from {\bf(A2)} (i) and \eqref{Phis}.
\end{proof}

 The following corollary is a direct conclusion of Theorem \ref{TsHar}, see \cite[1.4.2]{Wbook}.
\begin{cor}\label{densitys} Let the assumption of Theorem \ref{TsHar} hold. For any $T>0$, $\xi,\eta\in\mathbb{H}$, let $\beta(T,\eta)$ is defined in Theorem \ref{TsHar}.  Then $P_T(\xi,\cdot)$ is equivalent to $P_T(\xi,\cdot-\eta)$. Moreover,  for any $p>1$,
$$P_T\left\{\left(\frac{\d P_T(\xi,\cdot)}{\d P_T(\xi,\cdot-\eta)}\right)^{\frac{1}{p}}\right\}(\xi)\leq \exp\bigg[\ff{1}{2(p-1)} \beta(T,\eta)\bigg].$$
\end{cor}
\paragraph{Acknowledgement.} The authors would like to thank Professor Feng-Yu Wang for corrections and helpful comments.

\beg{thebibliography}{99}

\bibitem{B} K. Bahlali,  \emph{Flows of homeomorphisms of stochastic differential equations with measurable drift,}  Stochastic Rep. 67(1999), 53-82.



\bibitem{BWY} J. Bao, F.-Y. Wang, C. Yuan, \emph{Derivative formula and Harnack inequality for degenerate functionals SDEs,} Stoch. Dyn. 13(2013), 943-951.


\bibitem{BWY15} J. Bao, F.-Y. Wang, C. Yuan,  \emph{Hypercontractivity for Functional Stochastic Partial Differential Equations,} Electron. J. Probab. 20(2015), 1-15.




\bibitem{GW} A. Guillin, F.-Y. Wang, \emph{Degenerate Fokker-Planck equations: Bismut formula, gradient estimate and Harnack inequality,} J. Differential Equations, 253(2012), 20-40.

\bibitem{GM} L. Gy\"{o}ngy, T. Martinez,  \emph{On stochastic differential equations with locally unbounded drift,}  Czechoslovak Math. J. 51(2001), 763--783.

\bibitem{H} Xing Huang,  \emph{Strong Solutions for Functional SDEs with Singular Drift,}  to appear in Stochastic and Dynamics.

\bibitem{HW} X. Huang, F.-Y. Wang, \emph{Functional SPDE with Multiplicative Noise and Dini Drift,}  Ann. Fac. Sci. Toulouse Math., 6(2017), 519-537.


\bibitem{P} E. Priola, \emph{Pathwise Uniqueness for Singular SDEs driven by Stable Processes,} Osaka Journal of Mathematics, 49(2012), 421-447.




\bibitem{FYW} F.-Y. Wang,  \emph{Gradient estimates and applications for SDEs in Hilbert space with multiplicative noise and Dini continuous drift,}  J. Differential Equations, 260 (2016), 2792-2829.

\bibitem{Wbook} F.-Y. Wang, \emph{Harnack Inequality and Applications for Stochastic Partial Differential Equations,} Springer, New York, 2013.

\bibitem{W2} F.-Y. Wang, \emph{Hypercontractivity and Applications for Stochastic Hamiltonian Systems,} J. Func. Anal. 271(2017),  5360-5383.

\bibitem{WY} F.-Y. Wang, C. G. Yuan, \emph{Harnack inequalities for functional SDEs with multiplicative noise and applications,} Stoch. Proc. Appl. 121(2011), 2692-2710.

\bibitem{WZ1} F.-Y. Wang, X. C. Zhang, \emph{Derivative formula and applications for degenerate diffusion semigroups,} J. Math. Pures Appl., 99 (2013), 726-740.

\bibitem{WZ} F.-Y. Wang, X. C. Zhang,  \emph{Degenerate SDE with H\"{o}lder-Dini Drift and Non-Lipschitz Noise Coefficient,} SIAM J. Math. Anal. 48 (2016), 2189-2226.

\bibitem{WZ15} F.-Y. Wang, X. C. Zhang,  \emph{Degenerate SDEs in Hilbert Spaces with Rough Drifts,} arXiv:1501.0415.

\bibitem{YW} T. Yamada, S. Watanabe, \emph{On the uniqueness of solutions of stochastic differential equations,} J. Math. Kyoto Univ. 11(1971), 155-167.

\bibitem{Z1} X. C. Zhang, \emph{Stochastic flows and Bismut formulas for stochastic Hamiltonian systems,}
Stoch. Proc. Appl., 120(2010), 1929-1949.

\bibitem{Z} X. C. Zhang,  \emph{Strong solutions of SDEs with singural drift and Sobolev diffusion coefficients,}  Stoch. Proc. Appl. 115(2005), 1805-1818.


\bibitem{ZV} A. K. Zvonkin,  \emph{A transformation of the phase space of a diffusion process that removes the drift,}  Math. Sb. (1)93 (1974).
\end{thebibliography}

\end{document}